\documentclass[11pt]{amsart}
\title{Brill-Noether theory of curves on toric surfaces}
\author[Geoffrey Smith]{Geoffrey Smith}
\address{1 Huckleberry Rd., Hopkinton, MA, 01748}
\email{geoffrey.smith@yale.edu}
\usepackage{graphicx}
\usepackage{amsmath}
\usepackage{amsthm}
\usepackage{amsfonts}
\usepackage{comment}
\usepackage{amssymb}
\usepackage{eufrak}
\usepackage{mathtools}

\newtheorem{theorem}{Theorem}[section]
\newtheorem{lem}[theorem]{Lemma}
\newtheorem{cor}[theorem]{Corollary}
\newtheorem{prp}[theorem]{Proposition}
\newtheorem{proposition}[theorem]{Proposition}
\theoremstyle{remark}
\newtheorem{remark}{Remark}
\theoremstyle{definition}
\newtheorem{example}[theorem]{Example}
\newtheorem{definition}[theorem]{Definition}

\begin{document}
\maketitle
\begin{abstract}
A Laurent polynomial $f$ in two variables naturally describes a projective curve $C(f)$ on a toric surface. We show that if $C(f)$ is a smooth curve of genus at least 7, then $C(f)$ is not Brill-Noether general. To accomplish this, we classify all Newton polygons that admit such curves whose divisors all have nonnegative Brill-Noether number.
\end{abstract}

\section{Introduction and statement of results}
Let $C$ be a smooth projective curve of genus $g$ defined over a field $k$.
Brill-Noether theory is the study of certain special divisors on $C$ with unusually large complete linear systems, namely those divisors for which Riemann's inequality is strict. 
Let $D$ be a divisor on $C$ of degree $d$ and rank $r$. The \emph{Brill-Noether number} $\rho(g,d,r)$ associated to $D$ is then defined by the equation $\rho(g,d,r)=g-(r+1)(g-d+r)$. The Brill-Noether Theorem of Griffiths and Harris indicates that a general family of curves, the so-called \emph{Brill-Noether general} curves, have no divisors with negative Brill-Noether number. 

In this paper we examine a particular family of curves, namely the smooth projective curves defined by some Laurent polynomial $f\in k[x^{\pm 1},y^{\pm1}]$ in two variables on the toric surface associated to the Newton polygon of $f$. We determine that no such curve is Brill-Noether general if its genus is greater than six.  Let $f$ be a Laurent polynomial in two variables $x,y$, let $\Delta$ be the Newton polygon associated to $f$, $X_\Delta$ the toric surface associated to $\Delta$, and $V(f)$ the zero-set of $f$ in $(k^*)^2$. Then we will establish the following result:
\begin{theorem}\label{g<7}
Let $C$ be the closure of the image of $V(f)$ under the natural map $(k^*)^2\rightarrow X_\Delta$. If $C$ is smooth and has genus $g$ at least 7, then $C$ is not Brill-Noether general. Moreover, if $g\neq 10$, then there exists a divisor on $C$ of rank 1 with negative Brill-Noether number.
\end{theorem}
\begin{remark}
This bounds of this theorem are sharp; we will exhibit examples of curves $C$ without these special divisors for each genus $g$ with $g \leq 6$ .
\end{remark}

In fact, we prove a stronger result, and classify the interior Newton polygons of such general curves up to transformations of the lattice with the following, more technical theorem.
\begin{theorem}\label{classification}
There is an $f$ and $C$ associated to a Newton polygon $\Delta$ with no divisors having negative Brill-Noether number $\rho$ if and only if the convex hull of the interior lattice points  of $\Delta$ is empty or transforms under some invertible lattice transformation to one of the 11 polygons in Figure \ref{g<7figs}. Furthermore, if the interior Newton polygon is not equivalent under a lattice transformation to any of these 11 and is also not equivalent to the polygons in Figure \ref{g=10} or Figure \ref{3sigma}, then the gonality of $C$ is less than $\frac{g}{2}+1$. 
\end{theorem}
\begin{figure}[t]
\centering
\includegraphics[ clip]{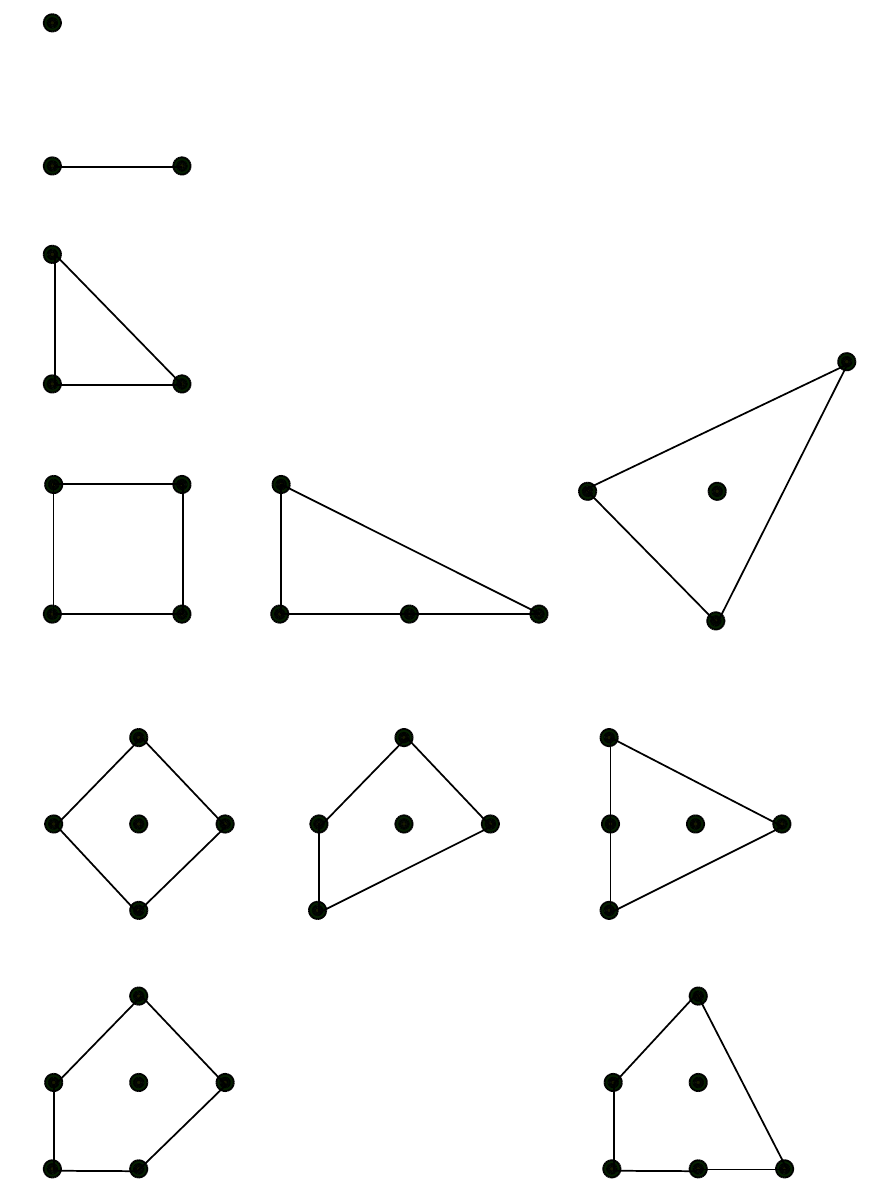}
\caption{Interior Newton polygons of Brill-Noether general curves}
\label{g<7figs}
\end{figure}
\begin{figure}[t]
\centering
\includegraphics[ clip]{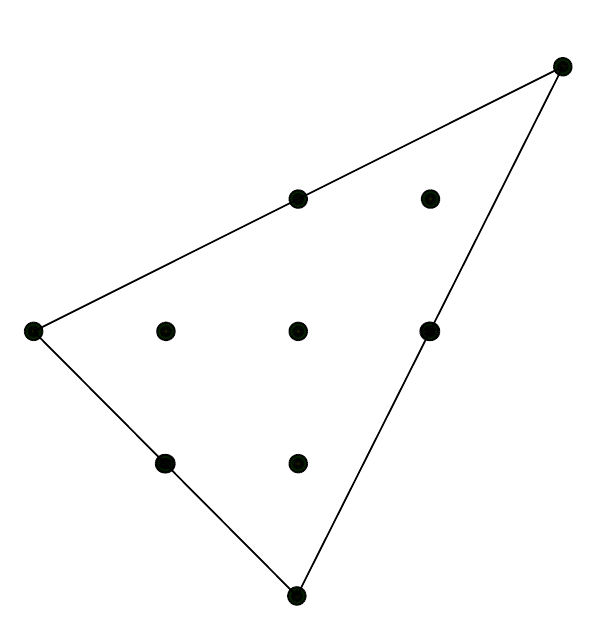}
\caption{Interior lattice polygon associated to some genus 10 curves of gonality 6}
\label{g=10}
\end{figure}
\begin{figure}[t]
\centering
\scalebox{.5}{\includegraphics[clip]{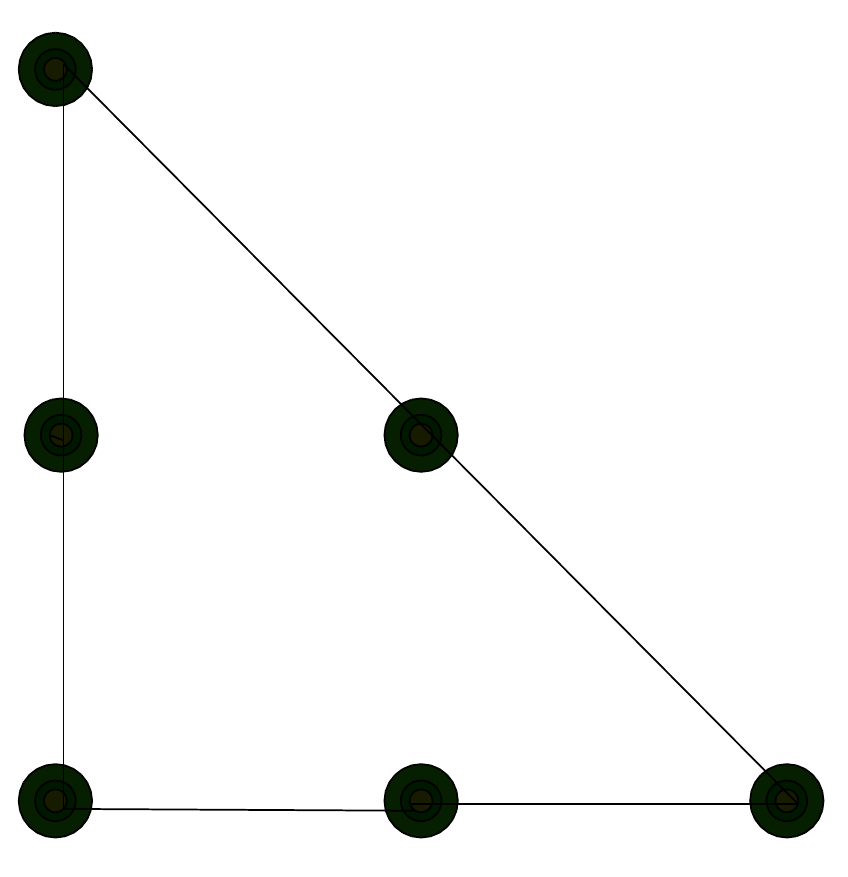}}
\caption{Six point interior polygon associated to no Brill-Noether general curves}
\label{3sigma}
\end{figure}
This paper is organized as follows. In Section 2 we introduce the  concepts and terminology from algebraic geometry and discrete geometry that will underlie the results in this paper. In Section 3 we use a bound on the lattice width of a curve's Newton polygon to establish the non-existence of curves without special divisors in the case $g>7$ and limit the set of possible interior Newton polygons to those of Figure \ref{g<7figs}. We also handle the exceptional genus 10 case in Proposition \ref{10}. In Section 4 we will then establish the existence of curves having no divisors with $\rho<0$ in the eleven cases indicated by Theorem \ref{classification}.

\section{Preliminaries}
Let $C$ be a smooth projective curve of genus $g$ defined over a field $k$. We refer to the dimension of the complete linear system of a divisor $D$ as the \emph{rank} of $D$. The Brill-Noether number of $D$, given by $\rho=g-(r+1)(g-d+r)$, then gives an expected dimension of the subscheme $W_d^r(C)$ of the Picard group of $C$ consisting of all degree $d$ divisor classes of rank at least $r$. 

The foundational result of Brill-Noether theory, the Brill-Noether Theorem of Griffiths and Harris, indicates that this expected dimension is correct for a general curve. We present a simplified version.
\begin{theorem}[\cite{GH80}]
Let $C$ be a smooth projective curve of genus $g$. Then the following statements are true:
\begin{enumerate}
\item If $\rho(g,d,r)\geq 0$, then there exists a divisor on $C$ of degree $d$ and rank $r$.
\item If $C$ is general and $\rho(g,d,r)<0$, then no divisor of degree $d$ on $C$ has rank $r$.
\end{enumerate}
\end{theorem} 

From this theorem follows immediately that on a general curve no divisor has negative Brill-Noether number. Despite this, many curves that are easy to describe, such as complete intersections in projective space, tend to have divisors with negative Brill-Noether number, and constructing particular curves of high genus having no such divisors requires surprising effort. Recently, there has been a proof using tropical methods by Cools et al. in \cite{CDPR12}.

\subsection{Lattice polygons and toric surfaces}
A \emph{lattice polygon} is defined to be a polygon in $\mathbb{R}^2$ whose vertices have integer coordinates. If $\Delta$ is a convex lattice polygon, then, following \cite{CC12}, we define the \emph{interior polygon} $\Delta^{(1)}$ of $\Delta$ as the convex hull of the set of lattice points contained in the interior of $\Delta$.

A \emph{lattice transformation} is an invertible linear map $f:\mathbb{Z}^2\rightarrow \mathbb{Z}^2$. They can be thought of as those transformations in $\mathrm{SL}_2(\mathbb{Z})$. This notion allows us to define the \emph{lattice width} of a polygon $\Delta$  robustly, as follows.
\begin{definition}
The \emph{lattice width} of a lattice polygon $\Delta$ is the minimal width of the image of $\Delta$ under a nontrivial linear map $\mathbb{Z}^2\rightarrow \mathbb{Z}$. 
\end{definition}
Lattice polygons in general are easy to analyze geometrically. One important result on general lattice polygons, both in general and in proving our main results, is the classical Pick's theorem:
\begin{theorem}\label{Pick's}
 If $\Delta$ is a lattice polygon, the area of $\Delta$ is given by $\frac{E}{2}+I-1$, where $E$ is the number of lattice points on the boundary of the polygon and $I$ is the number of points in the interior of $\Delta$. 
\end{theorem}

More recent results on lattice polygons include the following bound on the area of a convex lattice polygon in terms of its lattice width.
\begin{theorem}[\cite{TM74}, Theorem 2]\label{blackbox}
If $\Delta$ is a convex lattice polygon, then the area of $\Delta$ is at least $\frac{3}{8}\mathrm{lw}(\Delta)^2$.
\end{theorem}

Also, the lattice width of a lattice polygon's interior polygon determines its own lattice width. 
\begin{theorem}[\cite{CC12}, Theorem 4]\label{intlem}
The lattice width of $\Delta$ is given by 
\begin{equation*}
\mathrm{lw}(\Delta)=\mathrm{lw}(\Delta^{(1)})+2
\end{equation*}
unless $\Delta$ is equivalent under lattice transformation to the convex hull of $(0,0)$, $(0,d)$ and $(d,0)$ for some $d$, in which case $\mathrm{lw}(\Delta)=\mathrm{lw}(\Delta^{(1)})+3$.
\end{theorem}
\begin{figure}[t]
\centering
\includegraphics[scale=.5, clip]{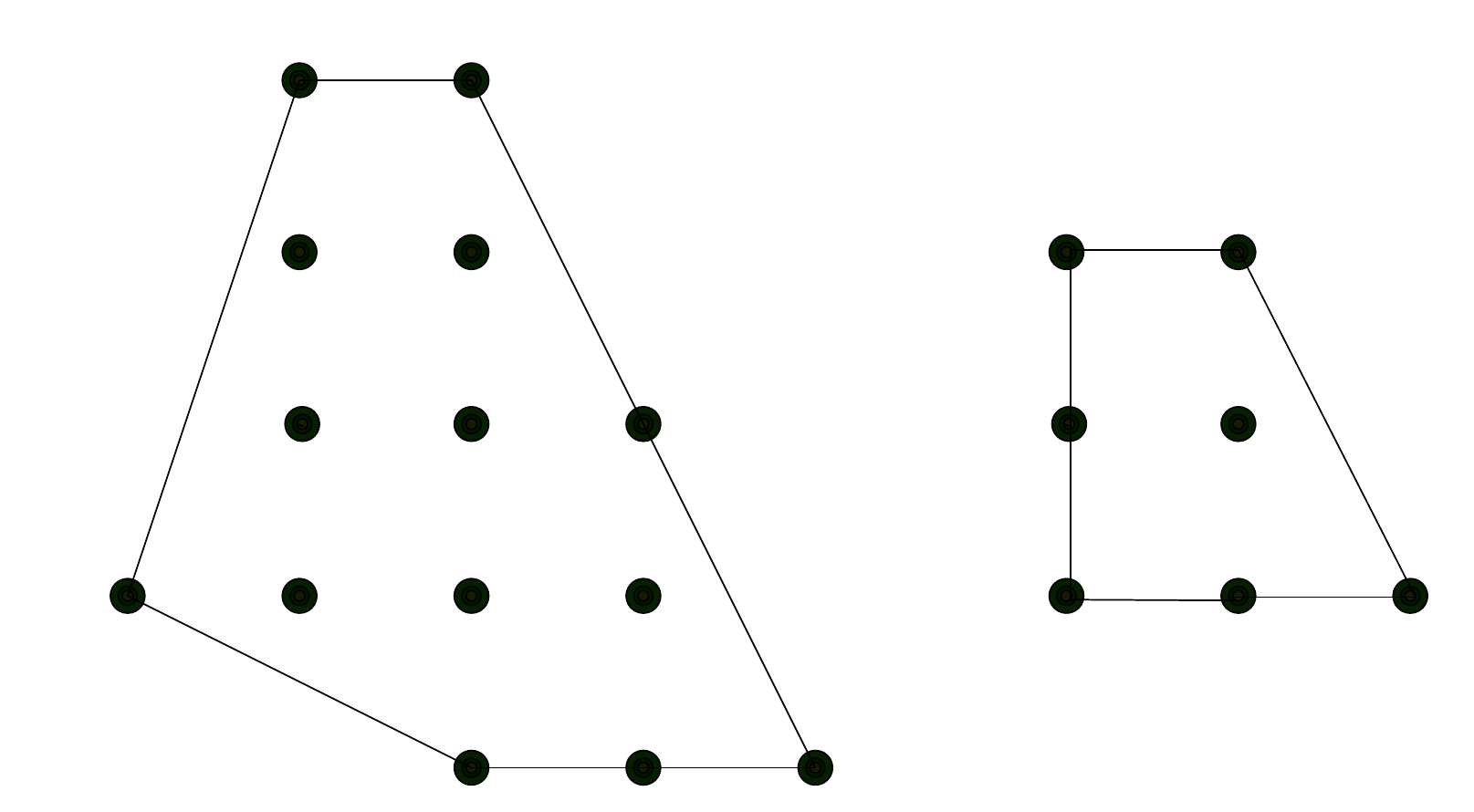}
\caption{A lattice polygon and its interior polygon. The former has lattice width 4, and the latter has lattice width 2.}
\label{InteriorEx}
\end{figure}

It is possible to associate to any convex lattice polygon a projective surface using the following construction, described in \cite{F93}:
\begin{definition}
If $\Delta$ is a convex lattice polygon containing $N$ lattice points, the \emph{toric surface associated to $\Delta$}, denoted $X_\Delta$, is the Zariski closure in $\mathbb{P}^{N-1}$ of the map $\phi:(k^*)^2\rightarrow \mathbb{P}^{N-1}$ given by $\phi(x,y)=(x^iy^j)_{(i,j)\in P}$.
\end{definition}
\begin{example}
Let $\Delta$ be the lattice triangle with vertices at $(0,0)$, $(1,0)$, and $(1,2)$, which thus also contains the interior point $(1,1)$. The toric surface $X_\Delta$ is then the closure of the image of $(k^*)^2$ under the map $(x,y)\mapsto (1:x:xy:xy^2)$, which is the surface in $\mathbb{P}^3$ defined by $x_2^2-x_1x_3=0$. 
\end{example}

Let $f=\sum_{i,j}a_{i,j}x^iy^j\in k[x^{\pm 1},y^{\pm 1}]$ be a nonzero Laurent polynomial, and let $V(f)$ be the curve of its zeros in $(k^*)^2$. The \emph{Newton polygon} of $f$ is defined as the convex hull in $\mathbb{R}^2$ of the points $(i,j)\in \mathbb{Z}^2$ with $a_{i,j}\neq 0$.

An important use of the Newton polygon of a curve is in calculating the genus of a smooth curve. Indeed, the following result holds.
\begin{theorem}[\cite{K77}, Section 4, Assertion 2]\label{Khovanskii}
Let $f\in k[x^{\pm 1},y^{\pm1}]$ be a Laurent polynomial with Newton polygon $\Delta$. Let $C$ be the closure of $V(f)$ in $X_\Delta$. If $C$ is smooth, then the genus of $C$ is the number of interior points of $\Delta$.
\end{theorem}

With the aid of this theorem, we can  give a proof of Theorem \ref{g<7} conditional on Theorem \ref{classification} now.
\begin{proof}[Proof of Theorem \ref{g<7}]
Let $f\in k[x^{\pm1},y^{\pm 1}]$ be a Laurent polynomial in two variables, let $\Delta$ be the Newton polygon of $f$, and let $C$ be the closure of $V(f)$ in $X_\Delta$. Suppose $C$ is smooth, and let $g$ be the genus of $C$ 
By Theorem \ref{Khovanskii}, we have that the interior polygon $\Delta^{(1)}$ of the Newton polygon $\Delta$ associated to $C$ contains $g$ points. But by Theorem \ref{classification} we also have that $\Delta^{(1)}$ is equivalent under some lattice transformation to one of the polygons in Figure \ref{g<7figs}, and every such lattice polygon contains at most six points, so we have $g\leq 6$. This proves the first half of the theorem.

Now suppose $C$ is a curve having no rank 1 divisors with negative Brill-Noether number. As before, by Theorem \ref{Khovanskii} we have that the interior polygon $\Delta^{(1)}$ of the Newton polygon associated to to $C$ has $g$ points, and by Theorem \ref{classification} $\Delta^{(1)}$ is equivalent to one of the lattice polygons in Figure \ref{g<7figs}, Figure \ref{g=10}, or Figure \ref{3sigma}. Since those polygons all have either at most six or ten points, we hence have $g\leq 6$ or $g=10$, as was to be shown.
\end{proof}

\section{Constructing special divisors}
In this section we prove one half of Theorem \ref{classification}; we show that if a Laurent polynomial $f$ in two variables with Newton polygon $\Delta$ describes a smooth curve $C$ having no divisors with negative Brill-Noether number, then $\Delta$ is equivalent under a lattice transformation to a polygon in Figure \ref{g<7figs}. To do so, we will bound the \emph{gonality} of these curves.
\begin{definition}
The \emph{gonality} of a curve $C$ is the minimal degree of a nonconstant rational map $C\rightarrow \mathbb{P}^1$.
\end{definition}
\begin{remark}
Because the zero-set of such a rational map is a divisor of rank 1, we can equivalently think of the gonality of $C$ as the minimal degree of a rank 1 divisor on $C$.  
\end{remark}
Using the structure of the Newton polygon we can bound the gonality of the curves we are investigating. For the rest of the section, let $f$ be a Laurent polynomial in two variables with Newton polygon $\Delta$, $V(f)$ the zero set of $f$ in $(k^*)^2$, and $C$ the closure of the image of $V(f)$ under the embedding $(k^*)^2\rightarrow X_\Delta$ of the torus in the toric surface associated to $\Delta$. 
\begin{lem}\label{lwgon}
The gonality of $C$ is at most the lattice width of $\Delta$.
\end{lem}
\begin{proof}
Suppose $\Delta$ has lattice width $w$ that without loss of generality is realized by the linear map $T:\mathbb{Z}^2\rightarrow \mathbb{Z}$ that sends $(u,v)$ to $u$. Then the map $\phi:V(f)\rightarrow \mathbb{P}^1$ given by $\phi:(x,y)\mapsto y$ has degree
 $w$, for $K(V(f))=\mathrm{Frac}(k[x,y]/(f))$ is generated over $k(y)$ by $x$, and $f$ can hence be expressed in the form $x^d f_d(y)+\ldots+x^{d+w}f_{d+w}(y)$for some integer $d$ and some Laurent polynomials $f_i$ in $y$. As such, dividing out $d$ factors of $x$ gives us a degree $w$ polynomial for $K(V(f))$ over $k(y)$, so the map $\phi$ has degree $w$ and $V(f)$ has gonality at most $w$. As such, since $C$ is birationally equivalent to $V(f)$, we have that $C$ has gonality at most $w$, as was to be shown.
\end{proof}
 We now use this fact and the properties of lattice polygons given in the previous section to prove a useful bound on the gonality.
\begin{theorem}\label{bound}
Suppose $C$ is smooth and has genus $g\geq 3$. Then the gonality of $C$ is at most $\sqrt{\frac{8}{3}(g-\frac{5}{2})}+2$.
\end{theorem}
\begin{proof}
Suppose first that $\Delta$ is not equivalent under lattice transformation to the convex hull of $(0,0)$, $(0,d)$ and $(d,0)$ for any $d$. Let $G$ be the gonality of $C$. We have by Lemma \ref{lwgon} that $G$ is at most the lattice width of $\Delta$ and hence (noting that we have excluded the exceptional case) Theorem \ref{intlem} implies that $G \leq \mathrm{lw}(\Delta^{(1)})+2$. Now, there are precisely $g$ lattice points within or on the edges of  $\Delta^{(1)}$ . We note that at most $g-3$ of those points can be interior for any $g\geq 3$, for if $\Delta^{(1)}$ is two-dimensional, it has at least three vertices, and if it is one-dimensional, it has zero interior vertices. As such, assuming $g \geq 3$, the area of $\Delta^{(1)}$ is at most $g-\frac{5}{2}$ by Pick's theorem. Therefore, using the lower bound for area provided by Theorem \ref{blackbox} we see that 
\begin{equation}
\frac{3}{8}\mathrm{lw}(\Delta^{(1)}(f))^{2}\leq g-\frac{5}{2}.
\end{equation}
As such, since $G\leq \mathrm{lw}(\Delta(f))\leq\mathrm{lw}(\Delta^{(1)}(f))+2$, we have:
\begin{equation}
\frac{3}{8}(G-2)^{2}\leq g-\frac{5}{2}.
\end{equation}
From this we conclude $G\leq \sqrt{\frac{8}{3}(g-\frac{5}{2})}+2$, as was to be shown.

Now, if $\Delta$ is equivalent to the convex hull of $(0,0)$, $(0,d)$, and $(d,0)$, then $X_\Delta$ is the projective plane and $C$ is a smooth plane curve of degree $d$. So $C$ has gonality $d-1$ and genus $g=\frac{(d-1)(d-2)}{2}$. Since $g\geq 3$, we hence have $d\geq 4$, and by direct calculation we have that $d-1\leq  \sqrt{\frac{8}{3}(\frac{(d-1)(d-2)}{2}-\frac{5}{2})}+2$, proving the result for all $C$.
\end{proof}

\begin{cor}\label{largeg}
If $C$ is smooth and has genus $g>12$, then $C$ has a divisor of rank 1 and with negative Brill-Noether number.
\end{cor}
\begin{proof}
Let $C$ be a smooth curve on which no rank 1 divisor has negative Brill-Noether number. As such, it has no divisors of rank 1 and degree $\lceil \frac{g}{2}\rceil$, and hence has gonality at least $\lceil\frac{g}{2}\rceil +1$. But then, from Theorem \ref{bound} we can conclude $\lceil\frac{g}{2}\rceil+1\leq \sqrt{\frac{8}{3}(g-\frac{5}{2})}+2$. This inequality is false when $g>12$, establishing the result.
\end{proof}
With this result much of the work of the classification is done; we know we can restrict our attention to those lattice polygons containing twelve or fewer interior points. When the lattice polygon also does not contain ten interior points, this is particularly simple.
\begin{proposition}\label{smallg}
The lattice width of a convex lattice polygon with $g$ interior points, where $7\leq g\leq 12$ and $g\neq 10$, is at most $\lceil\frac{g}{2}\rceil$.
\end{proposition}
\begin{proof}
When a convex lattice polygon $\Delta$ has 7, 8, 9 , 11, or 12 interior points, Table 1 of \cite{C12}--derived from exhaustively indexing the lattice polygons with $g$ interior points--indicates that $\Delta$ has maximum lattice width 4, 4, 5, 5, or 6 respectively; from this it follows that the lattice width is at most $\lceil\frac{g}{2}\rceil$ when $g$ is one of these values.
\end{proof}

As such, the only lattice polygons that can have an associated curve with no divisors having negative Brill-Noether number are those with ten or fewer than seven interior points. We now show that there is no such curve associated with a lattice polygon with ten interior points.

\begin{proposition}\label{10}
If $C$ is smooth and has genus 10, then there exists a divisor on $C$ with negative Brill-Noether number.
\end{proposition}
\begin{proof}
We have that the Newton polygon $\Delta$ associated to $C$ will have ten interior points. As indicated in section 5.9 of \cite{CC13}, a convex lattice polygon with 10 interior points has (up to lattice transformations) 22 possible interior lattice polygons. With the exception of the polygon displayed in Figure \ref{g=10}, all have lattice width at most 3, and none is equivalent to a lattice polygon that is the convex hull of the points $(0,0)$, $(0,d)$, and $(d,0)$ for some $d$, so by Theorem \ref{intlem} unless $\Delta^{(1)}$ is the polygon indicated in Figure \ref{g=10} we have that $\Delta$ has lattice width at most 5, so our associated curve $C$ will have gonality at most 5 by Lemma \ref{lwgon} and therefore some divisor of rank one having a negative Brill-Noether number. 

And if $\Delta^{(1)}$ is equivalent under some lattice transformation to the polygon in Figure \ref{g=10}, then without loss of generality that the Newton polygon of our $f$ is in fact that of Figure \ref{g=10}, so from the proof of Theorem 9 of \cite{CC12} we have that $C$ is the smooth intersection of two cubic surfaces in $\mathbb{P}^3$, namely $X_\Delta$ and a cubic surface described by $f$. Then $C$ will have gonality 6 generally, implying we cannot find a divisor of rank 1 with negative Brill-Noether number. Suppose that we have such a curve with no divisors of rank 1 having negative Brill-Noether number. By the Recognition Theorem of \cite{ELMS89}, such a curve will have a \emph{Clifford dimension} of 3, where the Clifford dimension is the minimal value of $r$ such that some divisor of rank $r$ has a degree $d$ that achieves the minimum value of $d-2r$ over all divisors on $C$, the so-called \emph{Clifford index}.  Therefore, we have that its Clifford index is at most 3, since otherwise some divisor of rank 1 would achieve it.  So there is some rank-3 divisor of degree at most 9 achieving the Clifford index, and such a divisor has negative Brill-Noether number. So every curve $C$ associated with $\Delta$ has a divisor with negative Brill-Noether number.
\end{proof}

\section{Curves of genus at most 6}
In this section we establish Theorem \ref{classification}; the partial results of Corollary \ref{largeg} and Propositions \ref{smallg} and \ref{10} together establish that no exceptional interior polygon has more than six points, so this process reduces to examining the interior polygons having at most six points.

First, we note that Figure \ref{g<7figs} lists up to lattice transformation every interior polygon having six or fewer vertices whose lattice width does not immediately imply the existence of divisors with negative Brill-Noether number in an associated curve, with the exception of the interior polygon in Figure \ref{3sigma}, which has associated curves that are plane curves of degree five and hence of Clifford index 1 and Clifford dimension 2. Therefore, any curve having Figure \ref{3sigma} as an interior polygon will have some rank 2 degree 5 divisor, and hence will have some divisor with negative Brill-Noether number.

As such, to establish Theorem \ref{classification} we now need only show that the 11 polygons in figure \ref{g<7figs} have associated curves that have no divisors with negative Brill-Noether number. A critical tool in our analysis will be the \emph{Riemann-Roch Theorem}, which we state here for completeness.
\begin{theorem}
If $K$ is the canonical divisor on a curve $C$ of genus $g$, and $D$ any divisor on the same curve, then we have 
\begin{equation*}
r(D)-r(K-D)=\mathrm{deg}(D)-g+1
\end{equation*}
where r(D) is the rank of $D$.
\end{theorem}
In practice, this means that given a divisor $D$ with negative Brill-Noether number, if $r(D)$ is sufficiently large then $K-D$ will be another divisor with negative Brill-Noether number but with smaller rank. We make this precise with the following proposition
\begin{prp}
If a curve of genus $g<5$ has a divisor with negative Brill-Noether number, it has a divisor $D$ with negative Brill-Noether number of rank 1. If a curve of genus 5 or 6 has a divisor with negative Brill-Noether number, it has a divisor $D$ with negative Brill-Noether number of rank 1 or 2.
\end{prp}
\begin{proof}
We illustrate the proposition in the case of a curve $C$ with $g=5$; the other cases are analogous. 
Suppose $C$ has a divisor $D$ of degree $d$ with negative Brill-Noether number and rank $r>2$. If $d\leq5$, we have that subtracting some set of $r-2$ points will produce a divisor $D'$ of degree at most 4 and rank 2, so $D'$ will have Brill-Noether number at most $5-3\cdot 3=-4$. Otherwise, by the Riemann-Roch theorem, we have $K-D$ has rank $r-d+4$ and has degree $8-d$, and so has Brill-Noether number $5-(r-d+5)(1+r)$. But we observe that $D$ has this same Brill-Noether number, so $K-D$ has negative Brill-Noether number and has degree at most 3, so by our previous argument $C$ has a divisor of rank at most 2 with negative Brill-Noether number.
\end{proof}

Based on this, it suffices to check that the gonality of curves having the Newton polygons in Figure \ref{g<7figs} will be at least $\frac{g}{2}+1$, and that the minimum degree of a rank two divisor on a curve with Newton polygon among the bottom five in Figure \ref{g<7figs} is at least $\lceil\frac{2g}{3}\rceil+2=6$. 

But the gonalities for curves associated to all eleven interior polygons are already known. In particular, in the proof of Theorem 4 of \cite{CC13} Castryck and Cools determine that the gonality of a certain general family of smooth curves $C$ on the toric surface associated their Newton polygon--those which are \emph{non-degenerate with respect to their Newton polygon}, with interior Newton polygon among those in Figure \ref{g<7figs}, all have gonality $\lceil\frac{g+1}{2} \rceil$. Hence, curves are associated with all eleven interior polygons that have no rank one divisors with negative Brill-Noether number.

And any curve $C$ associated to one of the bottom five interior polygons non-degenerate with respect to its Newton polygon has Clifford index 2 by Theorem 9 of \cite{CC13}, whence any divisor on $C$ of rank 2 has degree at least six and hence has negative Brill-Noether number.

So for each of the eleven interior polygons listed in Figure \ref{g<7figs} there is a family of associated curves having no divisors having negative Brill-Noether number, completing our proof of Theorem \ref{classification}.
\section*{Acknowledgments}
I would like to thank Sam Payne for suggesting the project and for guidance throughout the process. I also would like to thank Dustin Cartwright for many helpful conversations and Dhruv Ranganathan for his comments on drafts of this paper. Finally, I am grateful to the NSF for its support.
\bibliographystyle{plain}
\bibliography{Citations}
\end{document}